\documentclass[11pt, a4paper]{amsart}
\usepackage{amsmath,amsthm,amssymb,latexsym}
\usepackage[utf8]{inputenc}
\usepackage{graphicx}

\newcommand{\re}{\text{\rm Re\,}}
\newcommand{\im}{\text{\rm Im\,}}

\newcommand{\br}{{\mathbb{R}}}

\newcommand{\bz}{{\mathbb{Z}}}
\newcommand{\bc}{{\mathbb{C}}}

\newcommand{\csp}{{\mathcal{S}_p}}
\newcommand{\cspp}{{\mathcal{S}_{2p}}}

\renewcommand{\a}{\alpha}
\renewcommand{\b}{\beta}
\renewcommand{\l}{\lambda}
\newcommand{\s}{\sigma}

\newcommand{\dd}{\Delta}
\renewcommand{\o}{\omega}

\newcommand{\op}{{\omega}}
\newcommand{\g}{\gamma}
\renewcommand{\gg}{\Gamma}

\newcommand{\nt}{\noindent}

\newcommand{\bsl}{\backslash}
\newcommand{\ovl}{\overline}

\newcommand{\lp}{\left(}
\newcommand{\rp}{\right)}
\newcommand{\lt}{\left\{}
\newcommand{\rt}{\right\}}

\DeclareMathOperator{\di}{\rm d}

\allowdisplaybreaks
\numberwithin{equation}{section}

\newtheorem{theorem}{Theorem}[section]
\newtheorem{prop}[theorem]{Proposition}
\newtheorem{lemma}[theorem]{Lemma}
\newtheorem{corollary}[theorem]{Corollary}

\theoremstyle{definition}

\begin{document}

\title[Infinite band Schr\"odinger operators]
{On non-selfadjoint perturbations of  infinite band Schr\"odinger
operators and Kato method}
\author[L. Golinskii, S. Kupin]{L. Golinskii and S. Kupin}

\address{Mathematics Division, Institute for Low Temperature Physics and
Engineering, 47 Lenin ave., Kharkov 61103, Ukraine}
\email{golinskii@ilt.kharkov.ua}

\address{IMB, Universit\'e de Bordeaux, 351 cours de la Lib\'eration, 33405 Talence Cedex, France}
\email{skupin@math.u-bordeaux1.fr}

\date{\today}

\keywords{Schr\"odinger operators,  infinite band spectrum, Lieb--Thirring type inequalities,
resolvent identity, distortion}
\subjclass[2010]{Primary: 35P20; Secondary: 35J10, 47A75, 47A55}

\maketitle

\begin{center}
{\it Dedicated to Jean Esterle\\
on occasion of  his 70-th birthday}

\medskip
\end{center}

\begin{abstract}
Let $ H_0=-\dd+V_0 $ be a multidimensional Schr\"odinger ope\-rator with a real-valued potential 
and infinite band spectrum, and $H=H_0+V$ be its non-selfadjoint perturbation defined with the help of Kato approach.  We prove Lieb--Thirring type inequalities for the discrete spectrum of $H$ in the case 
when $V_0\in L^\infty(\br^d)$ and $V\in L^p(\br^d)$, $p>\max(d/2, 1)$.
\end{abstract}

\section*{Introduction}
\label{s0} The distribution of the discrete spectrum for a complex
perturbation of a model differential self-adjoint operator ({\it e.g.}, a Laplacian on $\br^d$,
a discrete Laplacian on $\bz^d$, etc.) were studied, for instance, in Frank--Laptev--Lieb--Seiringer \cite{fla}, Borichev--Golinskii--Kupin \cite{bgk}, Laptev--Safronov \cite{lasa2009}, Demuth--Hansmann--Katriel \cite{dhk}, Hansmann \cite{han2011} and Golinskii--Kupin \cite{gk2, gk3}.
Subsequent results in this direction can be found in Frank--Sabin \cite{frsa}, Frank--Simon \cite{frsi},  and Borichev--Golinskii--Kupin \cite{bgk2}. Similar techniques were applied to non-selfadjoint perturbations of other model operators of mathematical physics in Sambou \cite{sa1}, Dubuisson \cite{du1, du2}, Cuenin \cite{cue},  and Dubuisson--Golinskii--Kupin \cite{dgk}.

The present paper deals with the the case when the model self-adjoint Schr\"odinger operator with the bounded potential has an infinite band spectrum.
Consider a real-valued, measurable and bounded function $V_0$ on $\br^d$, $d\ge 1$, such that
the Schr\"odinger operator
\begin{equation}\label{e1}
H_0=-\dd+V_0
\end{equation}
is self-adjoint, $H_0^*=H_0$. We suppose throughout the paper that the
spectrum $\s(H_0)$ is {\it infinite band}, i.e.,
\begin{equation}\label{infband}
\s(H_0)=\s_{ess}(H_0)=I=\bigcup^\infty_{k=1} [a_k, b_k], \qquad a_k\to +\infty.
\end{equation}
With no loss of generality, it is convenient to assume,  that $a_1>0$. The gaps of the spectrum are called {\it relatively bounded} if
\begin{equation}\label{irelbound}
r=r(I):=\sup_k \frac{r_k}{b_k}<\infty,
\end{equation}
where  $r_k:=a_{k+1}-b_k$ is the length of $k$' gap in \eqref{infband}. For $d=1$, a generic example is a Hill operator with
a periodic potential (see \cite[Section XIII.16]{rs4}). It is well known (see \cite{maos})
that $r_k\to 0$ as $k\to\infty$ for potentials $V_0$ from $L^2$ on a period, and, consequently,  \eqref{irelbound}
obviously holds for these potentials.

Consider now
\begin{equation}\label{schrod}
H=H_0+V,
\end{equation}
where $V$ is a complex-valued potential, and the operator $H$ is defined by means of the Kato method. The Kato method is 
accepted nowadays to be the most powerful (compared to the operator and sum forms methods)
technique in abstract perturbation theory, see \cite{ka2, ge1}. It always works in
our cases of interest, and it is completely compatible with the operator and the form sums
whenever one or both of the latter are applicable.

A key feature of Kato's method is the following {\it resolvent identity}, see
\cite[Theorem 1.5, (2.3)]{ka2}, \cite[Lemma 2.2, (2.13)]{ge1}
\begin{equation}\label{resid}
R(z, H)=R(z, H_0)-\ovl{R(z,H_0)V_1}\cdot [I+\ovl{V_2R(z,H_0)V_1}]^{-1}\cdot V_2R(z,H_0), 
\end{equation}
where $R(z,T):=(T-z)^{-1}$ is the resolvent of a closed, linear operator $T$, $\ovl T$ denotes the operator closure of $T$, and $z$ lies in $\rho(H_0)\cap\rho(H)$, the intersection of the resolvent sets. 
Above,
\begin{equation}\label{factor}
V_1=|V|^{1/2}, \quad V_2={\rm sign} V\,|V|^{1/2}.
\end{equation}
As a matter of fact, in Kato's method the operator $H$ \eqref{schrod} is {\it defined} by
identity \eqref{resid}. If the difference $R(z,H)-R(z,H_0)$ is a compact operator at least for
one value of $z$, the celebrated theorem of H. Weyl (see, e.g., \cite[Corollary 11.2.3]{dav})
claims that $\s_{ess}(H)=\s_{ess}(H_0)$ and
$$ \s(H)=I\,\dot\cup\,\s_d(H) $$
where the discrete spectrum $\s_d(H)$ of $H$, i.e., the set of isolated
eigenvalues of finite algebraic multiplicity, can accumulate only on $I$. The symbol $\dot\cup$ stays for the disjoint union of two sets. 

The main purpose of this paper is to obtain certain quantitative information on the rate of the above accumulation to $\s_{ess}(H)$.
We require that the potentials at hand satisfy the following conditions:
\begin{equation}\label{hyp1}
V_0\in L^\infty(\br^d), \quad V\in L^p(\br^d), \quad p>\max(d/2, 1).
\end{equation}
Under these assumptions,  $H$ is a well-defined, closed and sectorial operator in $L^2(\br^d)$, and
$$ {\rm Dom}\,H={\rm Dom}\,H_0=W^{2,2}(\br^d). $$
Moreover, the resolvent difference appears to be compact.
Notice also, that for $V_0\equiv 0$ satisfying \eqref{hyp1}, the operator $H=-\dd+V$ is well defined as the form sum (cf. \cite[Section 6.1]{han2}).

Put $q:=1-d/2p>0$, and take $\o_0<0$ as
\begin{equation}\label{omega}
-\o_0=|\o_0|:=1+a_1+2\|V_0\|_\infty+(4\eta^2(p,d)\|V\|_p)^{1/q},
\end{equation}
see \eqref{tau} for the definition of the constant $\eta(p,d)$. Above, $a_1$ is the leftmost edge of $\s(H_0)$.

\begin{theorem}\label{t1} Let $H_0$ be an infinite band Schr\"odinger operator in $\br^d$, \newline $d\ge1$,
with relatively bounded spectrum $\eqref{infband}$-$\eqref{irelbound}$, and $V_0$, $V$
satisfy $\eqref{hyp1}$. Then, for $0<\tau<(q+1)p-1$
\begin{equation}\label{e104}
\sum_{z\in\s_d(H)} \frac{\di^p(z, I)}{(|\o_0|+|z|)^{d/2+\tau}} \le C(p,d,I,\tau)
\frac{\|V\|^p_p}{|\o_0|^\tau},
\end{equation}
where a positive constant $C(p,d,I,\tau)$ depends on $p,d,I$, and $\tau$.
\end{theorem}

\begin{corollary}\label{c1} Under assumptions of Theorem \ref{t1} we have
\begin{equation}\label{e1042}
\sum_{z\in\s_d(H)} \frac{\di^p(z,I)}{(1+|z|)^{d/2+\tau}}\le C'(1+|\o_0|)^{d/2}\,\|V\|^p_p
\end{equation}
and
\begin{equation}\label{e3031}
\sum_{z\in\s_d(H)} \frac{\di^p(z,I)}{(1+|z|)^{d/2+\tau}} \le C''
\bigl[(1+\|V_0\|_\infty)(1+\|V\|_p)\bigr]^{d/2q}\,\|V\|^p_p,
\end{equation}
where positive constants $C'=C'(p,d,I,\tau), \ C''=C''(p,d,I,\tau)$ depend on $p,d,I$, and $\tau$.
\end{corollary}

\section{Resolvent difference is in the Schatten--von Neumann class}
\label{s1}

The first key ingredient of the proof is the following result of Hansmann \cite[Theorem 1]{han1}.
Let $A_0=A_0^*$ be a bounded self-adjoint operator on the Hilbert space, $A$ a bounded operator with
$A-A_0\in\csp$, $p>1$. Then
\begin{equation}\label{hans}
\sum_{\l\in\s_d(A)} \di^p(\l, \s(A_0))\le K\|A-A_0\|_{\csp}^p,
\end{equation}
$K$ is an explicit (in a sense) constant, which depends only on $p$.

We are going to apply this result to
$$ A_0=A_0(\o)=R(\o,H_0), \quad A=A(\o)=R(\o,H),  
$$
where $\o\in\rho(H_0)\cap\rho(H)$ is an appropriate negative number. An operator-theoretic argument in this section gives the upper bound
of the right-hand side of \eqref{hans}. The lower bound for the left-hand side of \eqref{hans} is
obtained in the next section by using an elementary function-theoretic reasoning.

Generically, we are in the case $\o\le \o_0$, see \eqref{omega}.

\begin{lemma}\label{l2} Under assumptions $\eqref{hyp1}$ the following holds.
\begin{itemize}
\item[\it $i)$] for $\o<0$ and $q=1-(d/2p)>0$
\begin{equation}\label{e301}
\max\bigl(\|V_2R^{1/2}(\o,-\dd)\|_{\mathcal{S}_{2p}},\|\ovl{R^{1/2}(\o,-\dd)V_1}\|_{\mathcal{S}_{2p}}\bigr)
\le \eta(p,d) \lp\frac{\|V\|_p}{|\o|^q}\rp^{1/2}
\end{equation}
with
\begin{equation}\label{tau}
\eta(p,d):=\lt\frac{\gg\bigl(p-\frac{d}2\bigr)}{2^d\pi^{d/2}\gg(p)}\rt^{\frac1{2p}}.
\end{equation}

\item[\it $ii)$] for $\o<\o_0$,
\begin{equation}\label{e303}
\| \ovl{V_2 R(\o,H_0) V_1}\|\le \frac 12\,,
\end{equation}
and so
\begin{equation}\label{e3031}
\|(I+\ovl{V_2 R(\o,H_0) V_1})^{-1}\|\le 2.
\end{equation}

\item[\it $iii)$] for $\o<\o_0$,
\begin{equation}\label{e304}
||R(\o, H)-R(\o, H_0)||_\csp \le 4\eta^2(p,d)\, \frac{\|V\|_p}{|\o|^{q+1}}\,.
\end{equation}
\end{itemize}

\end{lemma}

\begin{proof}
To prove $i)$, write
$$ V_2R^{1/2}(\o,-\dd)=V_2(x)g_\o(-i\nabla), \quad g_\o(x)=(|x|^2-\o)^{-1/2}, \quad x\in\br^d. $$
By \cite[Theorem 4.1]{si1} (sometimes called the Birman--Solomyak (or Kato--Seilier--Simon) inequality)
\begin{equation}\label{birsol}
\|V_2R^{1/2}(\o,-\dd)\|_{\mathcal{S}_{2p}}\le (2\pi)^{-d/2p}\|V_2\|_{2p}\, \|g_\o\|_{2p}, \qquad p\ge 1.
\end{equation}
Of course, $\|V_j\|_{2p}=\|V\|^{1/2}_p$, and it is clear that
$$
\|g_\o\|^{2p}_{2p}= \|(|x|^2-\o)^{-1}\|^p_p=\frac1{|\o|^{p-d/2}}\,\int_{\br^d}\frac{dx}{(|x|^2+1)^p}<\infty
$$
for $p>\max({d}/2,1)$. The computation of the latter integral along with \eqref{birsol} gives
\begin{equation}\label{e307}
\|V_2R^{1/2}(\o,-\dd)\|_\cspp \le \eta(p,d) \lp \frac{\|V\|_p}{|\o|^q} \rp^{1/2}\,.
\end{equation}
The bound for $\|\ovl{R^{1/2}(\o,-\dd)V_1}\|_{\mathcal{S}_{2p}}$ is the same, since
$\ovl{R^{1/2}(\o,-\dd)V_1}=(V_1R^{1/2}(\o,-\dd))^*$.

Turning to $ii)$, let us begin with the following equality
\begin{equation}\label{re1}
R(\o,H_0)=R^{1/2}\bigl(I+R^{1/2}V_0R^{1/2}\bigr)^{-1}R^{1/2},
\end{equation}
where  $R:=R(\o,-\dd)$.
Indeed, it is clear that
\begin{equation}\label{re2}
\|R^{1/2}V_0R^{1/2}\|\le \frac{\|V_0\|_\infty}{|\o|}\le\frac12
\end{equation}
due to the choice of $\o_0$ \eqref{omega}, and so
\begin{equation}\label{re3}
\|\bigl(I+R^{1/2}V_0R^{1/2}\bigr)^{-1}\|\le2.
\end{equation}
Hence
\begin{equation*}
\begin{split}
\bigl(I+R^{1/2}V_0R^{1/2}\bigr)^{-1} &=\sum_{k=0}^\infty (-1)^k(R^{1/2}V_0R^{1/2})^k, \\
R^{1/2}\bigl(I+R^{1/2}V_0R^{1/2}\bigr)^{-1}R^{1/2} &=\sum_{k=0}^\infty (-1)^k R(V_0R)^k
=R\bigl(I+V_0R\bigr)^{-1}.
\end{split}
\end{equation*}
On the other hand, we have
$$ R(\o,H_0)=R-R(\o,H_0)V_0R, \quad R(\o,H_0)=R\bigl(I+V_0R\bigr)^{-1}. 
$$

Note that under assumption on $p$ \eqref{hyp1} ${\rm Dom} (-\dd)^{1/2}\subset {\rm Dom} V_2$
by the Sobolev embedding theorem, so we can write
$$ \ovl{V_2R(\o,H_0)V_1}=V_2R^{1/2}\bigl(I+R^{1/2}V_0R^{1/2}\bigr)^{-1} \ovl{R^{1/2}V_1}, $$
and $ii)$ follows from \eqref{e307}, \eqref{re3} and the choice of $\o_0$ \eqref{omega}.

To prove $iii)$, we apply the basic resolvent identity \eqref{resid}. In view of the Schatten norm version of H\"older's inequality (see, e.g., \cite[Theorem 2.8]{si1}) and \eqref{e3031}, we have
\begin{equation*}
\|R(\o,H)-R(\o,H_0)\|_\csp\le 2\|V_2R(\o,H_0)\|^2_\cspp.
\end{equation*}
Next, it follows from \eqref{re1} that
\begin{equation*}
V_2R(\o,H_0)=V_2R^{1/2}\bigl(I+R^{1/2}V_0R^{1/2}\bigr)^{-1}R^{1/2},
\end{equation*}
and so in view of \eqref{e307}
\begin{equation*}
\|V_2R(\o,H_0)\|^2_\cspp\le \frac4{|\o|}\,\|V_2R(\o,-\dd)\|_\cspp^2
\le 4\eta^2(p,d)\frac{\|V\|_p}{|\o|^{q+1}}\,.
\end{equation*}
The proof is complete.
\end{proof}

\section{Distortion for linear fractional transformations}
\label{s2}

To obtain the lower bound for the left-hand side of \eqref{hans}, we proceed with
the following distortion lemma for linear fractional transformations of the form
\begin{equation}\label{lft}
\l_\o(z):=\frac1{z-\o}\,, \quad \o\in\br.
\end{equation}
The proof of the below lemma is rather computational.


\begin{figure}[b]
\includegraphics[width=13cm]{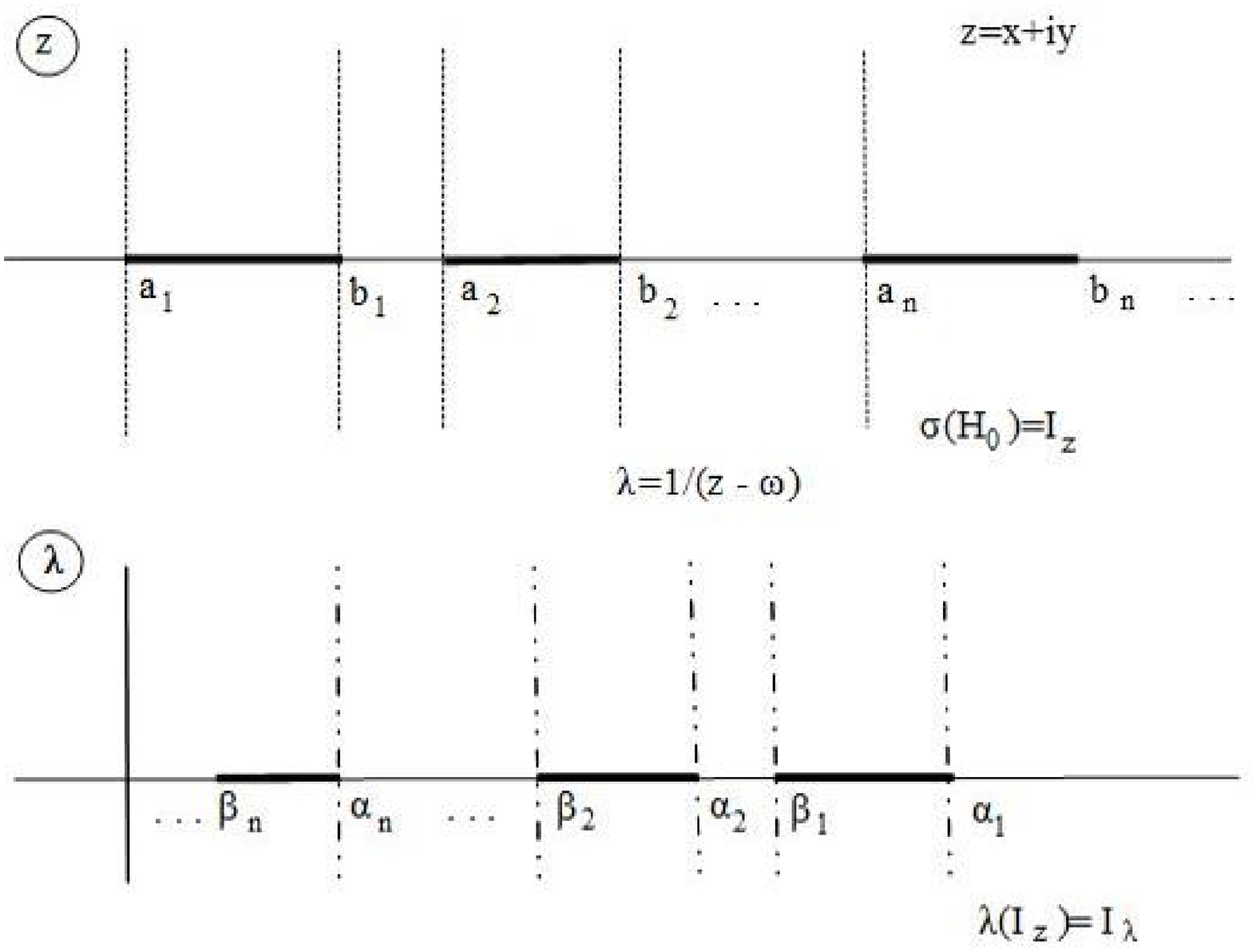}
\caption{Sets $I=\s(H_0)$ and $\l_\o(I)$ with map $\l_\o(z)=\frac 1{z-\o}$.}\label{f1}
\end{figure}

\begin{lemma}\label{l1}
Let
 \begin{equation}\label{infiniteband}
 I=I_z=\bigcup^{\infty}_{k=1}[a_k,b_k], \quad
 0<a_1<b_1<a_2<b_2<\ldots, \quad a_n\to+\infty,
 \end{equation}
and let $\l_\o(I)=I_\l$ be its image under the linear fractional transformation $\eqref{lft}$
\begin{equation*}
 \l_{\o}(I)=I_\l=\bigcup^{\infty}_{k=1}[\b_k(\o), \a_k(\o)], \ \ \b_k(\o)=\frac1{b_k-\o}\,, \ \  \a_k(\o)=\frac1{a_k-\o}\,.
\end{equation*}
Then for $\o<a_1$ the following bounds hold:

for $\re z<a_1$ or $\re z\in I$
\begin{equation}\label{distor1}
\frac{\di(\l_\o(z),\l_\o(I)}{\di(z,I_z)}>\frac1{3|z-\o|(|z-\o|+a_1-\o)}\,;
\end{equation}

for $b_k<\re z<a_{k+1}$, \ $k=1,2,\ldots$
\begin{equation}\label{distor2}
\frac{\di(\l_\o(z),\l_\o(I)}{\di(z,I_z)}
\ge\frac1{2|z-\o|^2}\,\left(1+\frac{a_{k+1}-b_k}{b_k-\o}\right)^{-1}\,.
\end{equation}
Moreover, if $\o<0$ and the gaps are relatively bounded $\eqref{irelbound}$,
then the unique bound is valid
\begin{equation}\label{distor3}
\frac{\di(\l_\o(z),\l_\o(I)}{\di(z,I_z)}\ge
\frac1{5(1+r(I))}\,\frac1{|z-\o|(|z-\o|+a_1-\o)}, \ \
z\in\bc\backslash I.
\end{equation}
\end{lemma}
\begin{proof}
Let us begin with the case $\o=0$ and put $\l_0=\l=z^{-1}$. If
$z=x+iy$ and $x=\re z\le0$, then $\re\l=x|z|^{-2}\le0$ and so
\begin{equation}\label{below1}
\frac{\di(\l,I_\l)}{\di(z,I_z)}=\frac{|\l|}{|z-a_1|}=\frac1{|z||z-a_1|}\ge
\frac1{|z|(|z|+a_1)}\,.
\end{equation}
Similarly, if $x\in I_z$, then $x\ge a_1$ and
$$ 0<\re\l=\frac{x}{|z|^2}\le\frac1{x}\le
a_1^{-1}=\a_1, \quad \di(\l,[0,\a_1])=|\im\l|=\frac{|y|}{|z|^2}\,.
$$
Since now $\di(z,I_z)=|y|$, we have
\begin{equation}\label{below2}
\frac{\di(\l,I_\l)}{\di(z,I_z)}\ge
\frac{\di(\l,[0,\a_1])}{\di(z,I_z)}=\frac1{|z|^2}>\frac1{|z|(|z|+a_1)}\,.
\end{equation}

\smallskip

Consider now the case when $x=\re z\notin I_z$. Fix $x$ in $k$'s
gap,
\begin{equation}\label{gapsforz}
b_k<x<a_{k+1}, \qquad k=k(x)=0,1,\ldots \end{equation} (we put
$b_0=0$ and treat $(b_0,a_1)$ as a number zero gap). Then
$$ \di(z,I_z)=\min(|z-b_k|,|z-a_{k+1}|), \quad k=1,2,\ldots, \quad \di(z,I_z)=|z-a_1|,
\quad k=0. $$
Define two sets of positive numbers
$$ u_j=u_j(x), \qquad v_j=v_j(x), \qquad j=k+1,k+2,\ldots $$ by equalities
\begin{equation*}
\re(\l(x+iu_j))=\frac{x}{x^2+u_j^2}=\a_j, \qquad
\re(\l(x+iv_j))=\frac{x}{x^2+v_j^2}=\b_j,
\end{equation*}
or, equivalently,
$$ u_j(x)=\sqrt{x(a_j-x)}, \qquad v_j(x)=\sqrt{x(b_j-x)}. $$
We also put $v_k=0$, so
$$ 0=v_k<u_{k+1}<v_{k+1}<u_{k+2}<v_{k+2}<\ldots, \quad u_n,\ v_n\to\infty, \ n\to\infty. $$

While the point $z$ traverses the line $x+iy$, $y\in\br$, its
image $\l(z)$ describes a circle with diameter $[0,1/x]$. We
discern the following two cases.

\smallskip
{\bf Case 1}. Assume that $\l$ lies over the ``gaps for $\l$''.
For each $k=0,1,\ldots$ there are two options for $\l$: interior
gaps
\begin{equation}\label{intgap}
\re\l\in(\a_{j+1},\b_j)\Longleftrightarrow v_j<|y|<u_{j+1}, \qquad
j=k+1,k+2,\ldots
\end{equation}
and the rightmost gap
\begin{equation}\label{rightgap}
\re\l\in(\a_{k+1},1/x)\Longleftrightarrow 0<|y|<u_{k+1}.
\end{equation}
For gaps \eqref{intgap} we have
 \begin{equation}\label{intgap1}
\di(\l,I_\l)=\min(|\l-\a_{j+1}|,
|\l-\b_j|)=\frac1{|z|}\min\left(\frac{|z-a_{j+1}|}{a_{j+1}},
\frac{|z-b_j|}{b_j}\right)\,.
\end{equation}
Define an auxiliary function $h$ on the right half-line
$$
h(t)=h(t,z):=\frac{|z-t|}{t}=\sqrt{\Bigl(\frac{x}{t}-1\Bigr)^2+y^2},
\quad t>0. $$
 Clearly, $h$ is monotone increasing on $(x,+\infty)$ and
 decreasing on $(0,x)$ with the minimum $h(x)=|y|$. Hence
 \eqref{intgap1} and \eqref{gapsforz} give
\begin{equation*}
\begin{split}
\di(\l,I_\l) &=\frac{\min(h(a_{j+1},z), h(b_j,z))}{|z|}\ge
 \frac{h(b_{j},z)}{|z|}\ge \frac{h(b_{k+1},z)}{|z|} \\ &\ge\frac{h(a_{k+1},z)}{|z|}=
 \frac{|z-a_{k+1}|}{a_{k+1}|z|}\,.
 \end{split}
 \end{equation*}
 Since by \eqref{gapsforz} $\di(z,I_z)\le |z-a_{k+1}|$, we see that
 \begin{equation}\label{bound1}
 \frac{\di(\l,I_\l)}{\di(z,I_z)}\ge \frac1{a_{k+1}|z|}\,.
 \end{equation}

For gaps \eqref{rightgap} let first $k\ge1$. Then as above in
\eqref{intgap1}
 \begin{equation*}
\di(\l,I_\l)=\frac1{|z|}\min\left(\frac{|z-a_{k+1}|}{a_{k+1}},
\frac{|z-b_k|}{b_k}\right)\,,
\end{equation*}
but it is not clear now which term prevails. If
$|z-a_{k+1}|\le|z-b_k|$ then $\di(z,I_z)=|z-a_{k+1}|$ and
$$
\frac{\di(\l,I_\l)}{\di(z,I_z)}=\frac1{|z|}\min\left(\frac1{a_{k+1}},
\frac{|z-b_k|}{b_k|z-a_{k+1}|}\right)=\frac1{a_{k+1}|z|}\,.
$$
Otherwise $|z-a_{k+1}|>|z-b_k|$ implies
$$
\frac{\di(\l,I_\l)}{\di(z,I_z)}=\frac1{|z|}\min\left(\frac1{b_{k}},
\frac{|z-a_{k+1}|}{a_{k+1}|z-b_{k}|}\right)\ge\frac1{a_{k+1}|z|}\,.
$$

Next, for $k=0$ one has $0<x<a_1$, and in case \eqref{rightgap}
$$ \di(\l,I_\l)=|\l-\a_1|=\frac{|z-a_1|}{a_1|z|}\,, \quad
\di(z,I_z)=|z-a_1|, $$ and so
\begin{equation}\label{kzero}
\frac{\di(\l,I_\l)}{\di(z,I_z)}=\frac1{a_1|z|}\,.
\end{equation}
 Finally, in the case of ``gaps for $\l$'' we come to the bound
\begin{equation}\label{gapsforl}
\frac{\di(\l,I_\l)}{\di(z,I_z)}\ge \frac1{a_{k+1}|z|}\,, \qquad
k=0,1,\ldots.
\end{equation}

A modified version of \eqref{gapsforl} will be convenient in the
sequel. For $k\ge1$ in view of $|z|\ge x>b_k$ we have
$$
\frac1{a_{k+1}|z|}\ge\frac{b_k}{a_{k+1}|z|^2} $$ and so for
$k=1,2,\ldots$
\begin{equation}\label{gapsforl1}
\frac{\di(\l,I_\l)}{\di(z,I_z)}\ge \frac1{|z|^2}\,
\left(1+\frac{a_{k+1}-b_k}{b_k}\right)^{-1}=\frac1{|z|^2}\left(1+\frac{r_k}{b_k}\right)^{-1}\,,
\end{equation}
$r_k=a_{k+1}-b_k$ is the length of $k$'s gap. Similarly, for $k=0$
one has from \eqref{kzero}
\begin{equation}\label{gapsforl2}
\frac{\di(\l,I_\l)}{\di(z,I_z)}\ge \frac1{|z|(|z|+a_1)}\,.
\end{equation}

\smallskip
{\bf Case 2}. Assume that $\l$ lies over the ``bands for $\l$''
\begin{equation}\label{bands1}
\re\l\in[\b_j,\a_{j}]\Longleftrightarrow u_j\le |y|\le v_{j},
\qquad j=k+1,k+2,\ldots.
\end{equation}
Now
$$ \di(\l,I_\l)=|\im\l|=\frac{|y|}{|z|^2}\,, $$
\begin{equation*}
\begin{split}
\di(z,I_z) &\le |z-a_{k+1}|\le
|y|+a_{k+1}-x=|y|+\frac{u_{k+1}^2}{x} \le
|y|\left(1+\frac{u_{k+1}}{x}\right)
\\ &=|y|\left(1+\sqrt{\frac{a_{k+1}-x}{x}}\right),
\end{split}
\end{equation*}
so that
\begin{equation}\label{bands2}
\frac{\di(\l,I_\l)}{\di(z,I_z)}\ge \frac1{|z|^2}\,
\left(1+\sqrt{\frac{a_{k+1}}{x}-1}\right)^{-1}\,.
\end{equation}

For $k\ge1$ (interior gap for $z$) inequality \eqref{bands2} can
be simplified in view of $x>b_k$
\begin{equation}\label{bands3}
\frac{\di(\l,I_\l)}{\di(z,I_z)}\ge \frac1{|z|^2}\,
\left(1+\sqrt{\frac{r_k}{b_k}}\right)^{-1}\,.
\end{equation}
Let now $k=0$, i.e., $0<x=\re z<a_1$. In our case
$\di(z,I_z)=|z-a_1|$ and
$$ |y|\ge u_1=\sqrt{x(a_1-x)}. $$
If $|y|\ge 2x$ then $|y|\ge \frac23|z|$ and so
\begin{equation}\label{bands4}
\frac{\di(\l,I_\l)}{\di(z,I_z)}=\frac{|y|}{|z|^2|z-a_1|}\ge
\frac23 \frac1{|z|(|z|+a_1)}\,.
\end{equation}
Otherwise, $|y|<2x$ implies
$$ 2\sqrt{x}>\sqrt{a_1-x}, \qquad x>\frac{a_1}5. $$
It follows now from \eqref{bands2} with $k=0$ that
\begin{equation}\label{bands5}
\frac{\di(\l,I_\l)}{\di(z,I_z)}\ge \frac1{3|z|^2}>\frac13
\frac1{|z|(|z|+a_1)}\,.
\end{equation}

We can summarize the results obtained above in the following two
bounds from below. A combination of \eqref{below1},
\eqref{below2}, \eqref{gapsforl2} and \eqref{bands5} gives
\begin{equation}\label{final1}
\frac{\di(\l,I_\l)}{\di(z,I_z)}>\frac1{3|z|(|z|+a_1)}\,, \qquad
\re z<a_1 \ \ {\rm or} \ \re z\in I_z.
\end{equation}
A combination of \eqref{gapsforl1} and \eqref{bands3} provides
\begin{equation}\label{final2}
\begin{split}
\frac{\di(\l,I_\l)}{\di(z,I_z)} &\ge\frac1{\g_k|z|^2}\,, \qquad
\g_k =\max\left\{1+\frac{r_k}{b_k},\
1+\sqrt{\frac{r_k}{b_k}}\right\}\,, \\ b_k&<\re z<a_{k+1}, \quad
k=1,2,\ldots.
\end{split}
\end{equation}
Since $\g_k<2(1+r_k/b_k)$, the latter can be written as
\begin{equation}\label{final3}
\frac{\di(\l,I_\l)}{\di(z,I_z)}
\ge\frac1{2|z|^2}\,\left(1+\frac{r_k}{b_k}\right)^{-1}\,, \quad
b_k<\re z<a_{k+1}, \quad k=1,2,\ldots.
\end{equation}

To work out the general case $\o\not=0$ and prove \eqref{distor1}
and \eqref{distor2}, it remains only to shift the variable and
apply the results just obtained to the shifted sequence of bands
$$ I_z(\o)=\bigcup_{k\ge1}[a_k-\o,b_k-\o]. $$
The final statement follows from a simple observation
$$ \frac{r_k}{b_k-\o}\le\frac{r_k}{b_k}\le r. $$
The proof is complete.
\end{proof}

\section{Lieb--Thirring type inequalities}
\label{s3}

We continue with Hansmann's inequality \eqref{hans} and the upper (lower) bounds for its right (left)
hand sides obtained in previous sections. In what follows,
$C_k=C_k(p,d,I)$, $k=1,2,\ldots$, denote positive constants, which depend on $p,d$, and the
set $I$ \eqref{infband}.

\begin{prop}\label{p1} Under conditions of Theorem \ref{t1}, for $\o\le\o_0$ defined
in $\eqref{omega}$, we have
\begin{equation}\label{e310}
\sum_{z\in\s_d(H)} \frac{\di^p(z,I)}{|z-\o|^p (|z-\o|+a_1-\o)^p}\le  C_1\,
\frac{\|V\|^p_p}{|\o|^{(q+1)p}}, \quad \o\le\o_0.
\end{equation}
\end{prop}

\begin{proof}
We apply \eqref{hans} with
$$ A_0=A_0(\o)=R(\o,H_0), \qquad A=A(\o)=R(\o,H).
$$
So, by Lemma \ref{l2}, with $I_\l=\l_\o(I)=\s(A_0)$
\begin{equation}\label{e102}
\begin{split}
&\sum_{\l\in\s_d(A)} \di^p(\l, I_\l) = \sum_{\l\in\s_d(A)} \di^p(\l, \s(A_0)) \\
&\le K\,\|R(\o, H)- R(\o, H_0)\|^p_\csp \le C_2\,\frac{\|V\|^p_p}{|\o|^{(q+1)p}}\,.
\end{split}
\end{equation}
Lemma \ref{l1} completes the proof of \eqref{e310} as
$$
\sum_{z\in\s_d(H)} \frac{\di^p(z,I)}{|z-\o|^p (|z-\o|+a_1-\o)^p}\le C_3(1+r(I))^p\,
\frac{\|V\|^p_p}{|\o|^{(q+1)p}}= C_4\,\frac{\|V\|^p_p}{|\o|^{(q+1)p}}.
$$
\end{proof}

\nt {\it Proof of Theorem \ref{t1}.}  The idea of the proof is to use the above proposition
and a ``convergence improving trick'' from \cite[p. 2754]{dhk}.

Put $\a:=p(q+1)-1-\tau>0$, and rewrite inequality \eqref{e310} in the form
$$
\sum_{z\in\s_d(H)} \frac{\di^p(z,I) \cdot s^\a}{|z+s|^p (|z+s|+a_1+s)^p}
\le C_1\,\frac{\|V\|^p_p}{s^{1+\tau}}, 
$$
where $s:=|\o|\ge s_0:=|\o_0|$. Observe that
$$ |z+s|\le |z|+s, \quad |z+s|+a_1+s\le 2s+|z|+a_1, $$
and so
$$
\sum_{z\in\s_d(H)} \frac{\di^p(z,I) \cdot s^\a}{(s+|z|)^p (2s+|z|+a_1)^p}
\le C_1\,\frac{\|V\|^p_p}{s^{1+\tau}}\,.
$$

Next, integrate the latter inequality with respect to $s$ from $s_0$ to infinity
and change the order of summation and integration
\begin{equation*}
\sum_{z\in\s_d(H)} \di^p(z,I)\, \int^\infty_{s_0} \frac{s^\a\, ds}{(s+|z|)^p (2s+|z|+a_1)^p}
\le  C_1\,\frac{\|V\|^p_p}{\tau s_0^\tau}.
\end{equation*}
The integral in the left-hand side converges, since
$$\a>0, \qquad 2p-\a-1=d/2+\tau>0. $$
Making the change of variables $s=(|z|+s_0)t+s_0$ and noticing that
$$ 2s+|z|+a_1= 2(|z|+s_0)t+2s_0+|z|+a_1\le 3(|z|+s_0)(t+1) $$
(see \eqref{omega}), we come to
\begin{equation*}
\begin{split}
\int^\infty_{s_0} \frac{s^\a\, ds}{(s+|z|)^p (2s+|z|+a_1)^p} &\ge
\frac1{3^p(|z|+s_0)^{2p-\a-1}}\,\int_0^\infty \frac{t^\a\,dt}{(1+t)^{2p}} \\ &=
\frac{C(p,d,\tau)}{(|z|+s_0)^{d/2+\tau}}\,,
\end{split}
\end{equation*}
which gives \eqref{e104}. The proof of the theorem is complete.
\hfill $\Box$

\medskip
\nt{\it Proof of Corollary \ref{c1}.}
Certainly,  $|z|+s_0\le (1+s_0)(1+|z|)$, and $(1+s_0)/s_0\le 2$. Hence
\begin{equation}\label{e312}
\begin{split}
\sum_{z\in\s_d(H)} \frac{\di^p(z,J)}{(1+|z|)^{d/2+\tau}} &\le C_2 (1+s_0)^{d/2}
\lp\frac{1+\o_0}{\o_0}\rp^\tau \|V\|^p_p \\ &\le C(p,d,I,\tau) (1+|\o_0|)^{d/2}\,\|V\|^p_p,
\end{split}
\end{equation}
as claimed. Furthermore, by \eqref{omega}
$$
1+s_0\le C_3 (1+\|V_0\|_\infty)(1+\|V\|^{1/q}_p)\le C_3 (1+\|V_0\|_\infty)(1+\|V\|_p)^{1/q},
$$
and inequality \eqref{e312} reads
\begin{equation*}
\sum_{z\in\s_d(H)} \frac{\di^p(z,I)}{(1+|z|)^{d/2+\tau}} \le C(p,d,I,\tau)\, (1+\|V_0\|_\infty)^{d/2}(1+\|V\|_p)^{d/2q}\, \|V\|^p_p.
\end{equation*}
It remains to note that $d/2<d/2q$, so \eqref{e3031} follows. The proof of corollary  is complete.
\hfill $\Box$

\medskip
Let us mention that inequality \eqref{e3031} is better (regarding the powers) than the corresponding
result in \cite[Theorem 0.2]{gk2}.

\medskip\nt
{\it Acknowledgments.} The authors are grateful to F. Gesztesy for a detailed account on the Kato method. They also thank M. Hansmann for valuable discussions on the subject of the article. The paper was prepared during the visit of the first author to the Institute of Mathematics of Bordeaux (IMB UMR5251) at University of Bordeaux. He  would like to thank this institution for the hospitality and acknowledges the financial support of ``IdEx of University of Bordeaux''.

\end{document}